\documentclass[a4paper,11pt]{article}
\usepackage{amsmath,amssymb,amsthm,fullpage,tikz,authblk}

\newtheorem{thm}{Theorem}[section]
\newtheorem{cor}[thm]{Corollary}
\newtheorem{lem}[thm]{Lemma}
\newtheorem{prop}[thm]{Proposition}

\newtheorem{quest}[thm]{Question}

\theoremstyle{remark}
\newtheorem*{rem}{Remark}

\newcommand{\eps}{\varepsilon}
\newcommand{\ie}{i.e.\ }
\newcommand{\zp}{\mathbb{Z}_+}
\newcommand*{\abs}[1]{\lvert #1\rvert}
\newcommand*{\prob}[1]{\mathbb{P}(#1)}
\newcommand*{\mean}[1]{\mathbb{E}(#1)}
\newcommand*{\ceil}[1]{\lceil #1\rceil}

\newcommand*{\bfrac}[2]{\genfrac{(}{)}{}{}{#1}{#2}}
\newcommand*{\conn}[1][G]{\operatorname{conn}(#1)}

\title{The number and average size of connected sets in graphs with degree constraints}
\author{John Haslegrave\thanks{Research supported by the UK Research and Innovation Future Leaders Fellowship MR/S016325/1.}}
\affil{Mathematics Institute, University of Warwick}

\begin{document}
\maketitle

\begin{abstract}The average size of connected vertex subsets of a connected graph generalises a much-studied parameter for subtrees of trees. For trees, the possible values of this parameter are critically affected by the presence or absence  of vertices of degree $2$. We answer two questions of Andrew Vince regarding the effect of degree constraints on general connected graphs. We give a new lower bound, and the first non-trivial upper bound, on the maximum growth rate of the number of connected sets of a cubic graph, and in fact obtain non-trivial upper bounds for any constant bound on the maximum degree. We show that the average connected set density is bounded away from $1$ for graphs with no vertex of degree $2$, and generalise a classical result of Jamison for trees by showing that in order for the connected set density to approach $1$, the proportion of vertices of degree $2$ must approach $1$. Finally, we show that any sequence of graphs with minimum degree tending to infinity must have connected set density tending to $1/2$.
\end{abstract}

\section{Introduction}\label{intro}
Connectedness is perhaps the most fundamental property of a network, and if nodes of a network may fail then the robust parts of the network, which remain connected even if all other nodes fail, are of particular interest. In particular, we might ask what a typical such part looks like. 

Such questions have a long history when the network is a tree, that is, a minimally connected graph. Here the connected subgraphs are precisely the subtrees. Jamison \cite{Jam83} studied the average order of a subtree of a fixed tree, and its normalised version, the \textit{subtree density}, that is, the average order of a subtree divided by the order of the host tree. He showed that, among host trees of given order, the subtree density is minimised for the path, where it is just over $1/3$. At the other extreme, the subtree density can be arbitrarily close to $1$ for large trees. Meir and Moon \cite{MM83} gave asymptotic results on the average value over all trees of a given order.

In both extremal cases, vertices of degree $2$ play a fundamental role. Jamison proved that in order for the subtree density to approach $1$, the proportion of vertices of degree $2$ in the original tree must also approach $1$. Much more is known about the maximum subtree density among trees of fixed order, and the possible structure of maximising trees; see recent work by Mol and Oellermann \cite{MO19} and by  Cambie, Wagner and Wang \cite{CWW21}. In the other direction, Jamison \cite{Jam83} asked whether any tree without vertices of degree $2$ (such trees being sometimes referred to as \textit{series-reduced trees}) must have subtree density exceeding $1/2$. This question remained unresolved for over 25 years, until Vince and Wang \cite{VW10} proved that the subtree density of a series-reduced tree must lie between $1/2$ and $3/4$. Both these bounds are best possible, and the present author \cite{Has14} classified the extremal sequences of trees in both cases.

The local subtree density, that of subtrees meeting a specified vertex, is also of interest. Jamison \cite{Jam83,Jam84} proved that this always exceeds the global density, and in fact showed monotonicity for a generalisation that interpolates between local and global. Again, vertices of degree $2$ play a key role; Wagner and Wang \cite{WW16} showed that the local subtree density of a tree at any vertex is at most twice its subtree density, but for series-reduced trees the two can differ by at most a constant.

There are at least two natural generalisations of subtree density to the case where the underlying graph, $G$, is connected but not necessarily a tree. One is to ask about the average order of subtrees contained in $G$, that is, subgraphs which are trees, and define the subtree density of general graphs exactly as for trees.  The other is to ask about connected induced subgraphs as alluded to above. In other words we say a nonempty set of vertices of $G$ is a \textit{connected set} if the subgraph consisting of these vertices and all edges between them is connected, and the \textit{connected set density} is the normalised average order of the connected sets of $G$. We remark that the profile of the connected sets of $G$ will typically be very different to that of its subtrees, since every connected set corresponds to at least one subtree on the same set of vertices, but larger connected sets typically correspond to a greater number, biassing the average order of a subtree upwards.

Both settings have been the focus of much recent work. The study of the subtree density of a general graph was initiated by Chin, Gordon, MacPhee and Vincent \cite{CGMV}, and the connected set density by Kroeker, Mol and Oellermann \cite{KMO18}. Both generalisations appear to be much more difficult than the tree case. The path is believed to minimise both parameters, which would generalise Jamison's initial result for trees, and this was explicitly conjectured in the case of connected set density \cite{KMO18,Vin21}. This conjecture was proved only very recently, independently and almost simultaneously, first by Vince \cite{Vin21+a} and secondly by the present author \cite{Has21+}. For the subtree density it is still not known, and likewise the subtree density is conjectured to be maximised by the complete graph, but this is not known. A key difficulty is the lack of monotonicity in these parameters under edge addition; see recent work by Cameron and Mol \cite{CM21} for a discussion of this. (While they consider subtree density, the same examples show that monotonicity fails, to the same degree, for connected set density.)

In this paper we consider exclusively the connected induced subgraph setting. Fix a connected graph $G$ on $n$ vertices. All graphs considered will be simple, finite and connected, and we shall frequently assume that $G$ is non-trivial (\ie has at least two vertices). A non-empty set $S\subseteq V(G)$ is a \textit{connected set} for $G$ if the induced subgraph $G[S]$ is connected. Writing $\conn$ for the set of connected sets for $G$, let $N(G):=\abs{\conn}$ be the number of these sets, $A(G):=N(G)^{-1}\sum_{S\in\conn}\abs{S}$ be their average order, and $D(G):=A(G)/n$ be the connected set density.

We consider two questions posed by Andrew Vince \cite{Vin21}. The first concerns the number of connected sets. Write $c(G)=\frac12\sqrt[n]{N(G)}$, so that $N(G)=(2c(G))^n$. Clearly $1/2<c(G)<1$ for any graph $G$, and without further restrictions both bounds are best possible, even for trees, since $N(P_n)=(n^2+n)/2=(1+o(1))^n$ and $N(K_{1,n-1})=2^{n-1}+n-1=(2-o(1))^n$.
\begin{quest}[{\cite[Question 5]{Vin21}}]\label{q5}Let $\mathcal G_3$ be the set of finite connected cubic graphs. What is $c_3:=\limsup_{G\in \mathcal G_3}c(G)$?
\end{quest}
Vince gave the lower bound $c_3>0.777$, arising from the prisms or circular ladders. We give an improved lower bound of $c_3>\sqrt[4]{7/16}>0.813$.

We also give the the first non-trivial upper bound on $c_3$, that $c_3<0.9781$. As might be expected, the upper bound we obtain applies to subcubic graphs (that is, graphs of maximum degree at most $3$), not just cubic graphs. Our methods also give non-trivial upper bounds on $c_d$, defined analogously, for each fixed $d$. However, no weaker bound on the maximum degree can be sufficient: there exist graph sequences with $c(G)\to1$ but with $\Delta(G)$ growing arbitrarily slowly as a function of $\abs{G}$.

The other question we consider concerns the connected set density $D(G)$ for graphs with lower bounds on the degrees.
\begin{quest}[{\cite[Question 8]{Vin21}}]\label{q8}Do there exist graphs with minimum degree at least $3$ and connected set density arbitrarily close to $1$?
\end{quest}
We answer this in the negative, as a consequence of some stronger results. First, we bound the connected set density of graphs with minimum degree at least $3$ away from $1$; in fact the same upper bound applies to graphs with no vertex of degree $2$, whether or not vertices of degree $1$ are permitted. Secondly, we show that in any sequence of graphs with connected set density tending to $1$, the proportion of vertices of degree $2$ must tend to $1$; this extends a classical result of Jamison for trees \cite{Jam83}. Finally, we give upper and lower bounds on the density of graphs with minimum degree at least $d$, which tend to $1/2$ as $d\to\infty$. This complements very recent work by Vince \cite{Vin21+b}, giving explicit, stronger lower bounds under the more restrictive assumption of $k$-connectivity.

\section{The number of connected sets of cubic graphs}\label{sec:q5}
We first give an improved lower bound on $c_3$. Vince \cite[Lemma 7.2]{Vin21} showed that the prism graph, or circular ladder, $\Pi_n$, satisfies 
\[N(\Pi_n)=1-3n+(1+\sqrt2)^n+(1-\sqrt2)^n+3n\frac{(1+\sqrt2)^n-(1-\sqrt2)^n}{2\sqrt2}.\]
In particular, $N(\Pi_n)\sim \frac{3n(1+\sqrt2)^n}{2\sqrt2}$ and so, since $\abs{\Pi_n}=2n$, we have 
\[\lim_{n\to\infty}c(\Pi_n)=\frac12 \sqrt{1+\sqrt 2}\approx 0.777.\]
For $n\geq 2$, the \textit{criss-cross prism} $\Pi^\times_n$ is the cubic graph consisting of two cycles $v_1v_2\cdots v_{2n}$ and $w_1w_2\cdots w_{2n}$ together with the edges $v_iw_{i+1}$ for $i$ odd and $v_iw_{i-1}$ for $i$ even; see Figure \ref{criss-cross}.

\begin{figure}
\centering
\begin{tikzpicture}
\foreach \x in {0,60,...,300}{\draw[fill] (\x:2) circle (0.05);
\draw (\x:2) -- (\x+60:2);
\draw[fill] (\x:3) circle (0.05);
\draw (\x:3) -- (\x+60:3);}
\foreach \x in {0,120,240}{\draw (\x:2) -- (\x+60:3);
\draw (\x+60:2) -- (\x:3);}
\end{tikzpicture}
\caption{The criss-cross prism $\Pi_3^\times$ (the Franklin graph).}\label{criss-cross}
\end{figure}
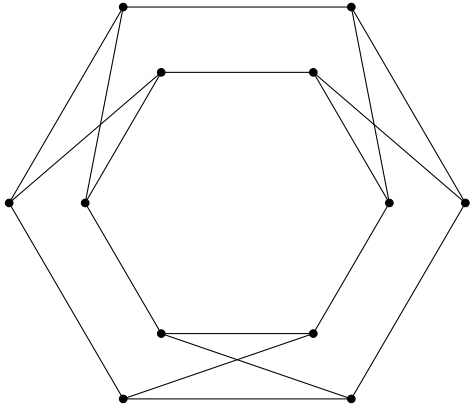
The following result gives an exact value for $N(\Pi_n^\times)$ provided $n\geq 3$. For $n=2$ we have $\Pi_2^\times\cong\Pi_4$ and so $N(\Pi_2^\times)=167$.
\begin{thm}\label{thm:criss-cross}For $n\geq 3$ we have $N(\Pi_n^\times)=7^n+\frac {191}3n7^{n-2}-\frac{16}{3}n$. In particular, since $\abs{\Pi_n^\times}=4n$ we have $\lim_{n\to\infty}c(\Pi_n^\times)=\frac12\sqrt[4]{7}$.
\end{thm}
\begin{proof}
For each $i$ with $1\leq i\leq n$, write $V_i=\{v_{2i-1},v_{2i},w_{2i-1},w_{2i}\}$ and $E_i=\{v_{2i}v_{2i+1},\allowbreak w_{2i}w_{2i+1}\}$ (where we take subscripts modulo $2n$). Fix a set $A\subseteq V(\Pi_n^\times)$. We say $A$ is ``good'' for $E_i$ if $A$ contains both vertices of at least one edge in $E_i$, and is ``bad'' for $E_i$ otherwise. Whether $A$ is good for $E_i$ depends on $A\cap\{v_{2i},v_{2i+1}, w_{2i},w_{2i+1}\}$; there are $7$ choices of this intersection for which $A$ is good for $E_i$.

Suppose that $A$ is connected, but for some $i\neq j$ we have $A$ is bad for both $E_i$ and $E_j$. Since $E_i\cup E_j$ separates $V_{i+1}\cup\cdots\cup V_j$ from $V_{j+1}\cup\cdots\cup V_i$ (again, taking subscripts modulo $2n$), $A$ intersects exactly one of them. It follows that the values of $i$ for which $A$ intersects $V_i$ form an interval modulo $2n$. Further, if $A$ intersects every $V_i$ then it is good for all but at most one of the $E_i$, whereas otherwise  $A$ is good for $E_i$ if and only if it intersects both $V_i$ and $V_{i+1}$. 

Notice that if $A$ is good for $E_i$ for every $i$, then $A$ is necessarily connected. This is because for any choice of one edge from each $E_i$, the vertices of these edges induce a cycle in $\Pi_n^\times$, and every other vertex is adjacent to a vertex of this cycle. Similarly, taking one edge from each of $E_1,\ldots, E_i$ for $i\leq n-1$, the vertices of these edges induce a path, and every vertex of $V_2\cup\cdots\cup V_{i-1}$ is adjacent to or on the path, so if $A$ intersects precisely $V_1,\ldots,V_{i+1}$ and is good for $E_1,\ldots,E_i$ then connectedness of $A$ depends only on how $A$ intersects $\{v_1,v_2,v_3,w_1,w_2,w_3\}$ and $\{v_{2i},v_{2i+1},v_{2i+2},w_{2i},w_{2i+1},w_{2i+2}\}$. In particular, the intersection with the former cannot be $\{v_2,v_3,w_2\}$ or $\{v_2,w_2,w_3\}$, and similarly for the latter. 

It follows that for $2\leq i\leq n-2$ the number of connected sets which are good for precisely $E_1,\ldots,E_i$ is $4\times4\times7^i-2\times4\times 2\times7^{i-1}+2\times 2\times 7^{i-2}=676\times 7^{i-2}$. For $i=1$ the two ends are not independent, and there are $95$ connected sets. For $i=n-1$ we have fewer choices for the ends since we require the set not to be good for $E_n$, and there are $333\times 7^{n-3}$ connected sets. Each of these sets corresponds to $n$ connected sets obtained by rotation. Finally, there are $7^n$ connected sets which are good for all the $E_i$, and $13n$ which are bad for all the $E_i$
(those entirely contained in some $V_i$). Summing all these terms gives the required formula.
\end{proof}
\begin{rem}We may likewise calculate the total order of these connected sets, to obtain 
\[D(\Pi_n^\times)=\frac{\frac{3820}{3}n^27^{n-3}-\frac{292}{9}n7^{n-3}+\frac{320}{9}n}{4n(7^n+\frac {191}{3}n7^{n-2}-\frac{16}{3}n)}=\frac57-o(1).\] 
We omit the details since bounding the connected set density is not the aim of this section.
\end{rem}

Next we prove an upper bound on the number of connected sets of an $n$-vertex cubic graph. Note that $c(G)=\sqrt[n]{N(G)/2^n}=\sqrt[n]{\prob{S\text{ connected}}}$, where $S$ is a subset of $V(G)$ chosen uniformly at random, \ie each vertex is independently chosen to be in $S$ with probability $1/2$. (Recall that we do not consider $\varnothing$ to be a connected set, although it is a possible choice of $S$.) Our approach will be to bound this probability by $\prob{I(S)=0}$ where $I(S)$ is the number of isolated vertices in the subgraph induced by $S$. Since there are $n$ connected sets which induce an isolated vertex, but there are at least $e(G)(e(G)-13)/2$ four-vertex sets which are not connected but have no isolated vertex, we have $\prob{S\text{ connected}}\leq\prob{I(S)=0}$ for $n\geq 10$.

We will bound this probability by revealing vertices one-by-one in an adaptive order. To that end, we first prove a general lemma which applies to similar revelation processes. 

\begin{lem}\label{lem:game}Suppose $X$ is a fixed random variable which takes non-negative integer values, is absolutely bounded and satisfies $\prob{X=0}<1$. Let $X_1,\ldots$ be a sequence of independent copies of $X$, and let $p_n$ be the probability that we have $\sum_{i=1}^kX_i\geq n$ before $X_k=0$. Then $p_n=\Theta(z^{-n})$, where $z$ is the unique positive solution of $\mean{z^X}=1+\prob{X=0}$.\end{lem}
\begin{proof}Write $A$ for the event that we have $\sum_{i=1}^kX_i\geq n$ before $X_k=0$. Suppose that the maximum possible value of $X$ is $r$. Then, for every $n\geq r$, we have $\prob{A\mid X_1=0}=0$ and $\prob{A\mid X_1=i}=p_{n-i}$ for each $i>0$, giving $p_n=\sum_{i=1}^r\prob{X=i}p_{n-i}$. (Here we take $p_0=1$.)

Existence and uniqueness of $z$ follow since the probability generating function is continuous, strictly increasing, unbounded, and takes the value $1\leq 1+\prob{X=0}$ at $1$. Set $q_n=z^np_n$ for each $n$, and note that $q_n>0$. Let $c_1=\min_{0\leq i<r}q_i$ and $c_2=\max_{0\leq i<r}q_i$. We show by induction that $c_1\leq q_n\leq c_2$ for all $n$; this is true by definition for $n<r$. If $n\geq r$ we have $q_n=\sum_{i=1}^r\prob{X=i}z^iq_{n-i}$; by the induction hypothesis \[c_1\sum_{i=1}^r\prob{X=i}z^i\leq q_n\leq c_2\sum_{i=1}^r\prob{X=i}z^i.\]
The result follows since, by choice of $z$, we have $\sum_{i=1}^r\prob{X=i}z^i=\mean{z^X}-\prob{X=0}=1$.
\end{proof}
We wish to apply this in a setting where the variables $X_i$ are not necessarily independent or identically distributed, but in some sense $X$ gives the worst-case distribution. We make this precise in the following corollary.
\begin{cor}\label{cor:game}Suppose that $X$ is as above, and $X_1,\ldots$ is a sequence of random variables with the following properties: $X_1$ is absolutely bounded with $\prob{X_1=0}<1$; and for each $k\geq 2$ we have $X_k\mid\mathcal F_{k-1}\leq_{st} X$, where $\mathcal F_{k-1}$ is the $\sigma$-algebra generated by $X_1,\ldots,X_{k-1}$. Then, defined as above, $p_n=O(z^{-n})$.
\end{cor}
\begin{proof}Let $p'_n$ be the probability that $\sum_{i=2}^kX_i\geq n$ before $X_i=0$ for some $2\leq i\leq k$ (\ie the equivalent probability for the process starting with $X_2$ rather than $X_1$). For $n$ at least the maximum value of $X_1$, we obtain $p_n$ as the weighted average of $p'_{n-i}$ for a constant set of values of $i$. Thus it is sufficient to show $p'_n=O(z^{-n})$. The stochastic dominance condition implies that we may couple each $X_k$ for $k\geq 2$ with a variable $X'_k$ such that $X_k\leq X'_k$ and $X'_k\mid \mathcal F_{k-1}\overset{d}{=}X$; clearly we may do this without introducing additional dependence, meaning that $X'_2,\ldots$ are independent. Now if $\sum_{i=2}^kX_i\geq n$ before $X_i=0$ for some $2\leq i\leq k$ then clearly the same applies to the $X'_i$, and so Lemma \ref{lem:game} applied to the latter gives an upper bound on $p'_n$.
\end{proof}
For $d\geq 2$ let $\mathcal{G}_d^-$ be the family of connected graphs with all degrees at most $d$.
\begin{thm}We have $c_d\leq\limsup_{G\in \mathcal G_d^-}c(G)\leq y_d$, where $y_d<1$ is the unique positive root of $2^{d+1}y^{2d-1}-2^dy^{2d-2}+1-2^d=0$.\end{thm}
\begin{proof}It is sufficient to show that $\prob{I(S)=0}=O(y_d^n)$, where $G\in\mathcal G_d^-$ has $n$ vertices and $S$ is a uniformly random subset of $V(G_n)$. We bound this probability by revealing whether or not vertices are in $S$ one by one, stopping if we have revealed that some vertex is in while all its neighbours are out or if we have revealed enough to know that no such vertex exists. We mark a vertex as ``safe'' if either it is known to be out or it has a neighbour which is known to be in; we say that a vertex is ``unsafe'' if it has not yet been marked safe.

We reveal vertices in stages according to the following strategy. At the start of each stage, if all vertices are safe, end the process with success. Otherwise choose an unsafe vertex $v$, and if possible choose a $v$ with at least one safe neighbour. By connectivity, this is possible provided there is a safe vertex, which will be the case in every stage except the first. If $v$ has $r>0$ unrevealed neighbours, let them be $w_1,\ldots, w_{r}$. We will ensure that at the start of each stage all revealed vertices will be marked safe, so that our choice of $v$ is necessarily unrevealed. 

First reveal $v$; if it is not in $S$, mark $v$ safe and move on to the next stage. If $v$ is in $S$, mark all its neighbours as safe and reveal $w_1,\ldots,w_{r}$ one by one; if any of these is in $S$, stop revealing immediately, mark its neighbours as safe, and proceed to the next stage. If none of them are in $S$ (including the case $r=0$), then $v$ is isolated, so end the process with failure. Note that if we proceed to the next stage, the only vertices revealed to be in $S$ during this stage are $v$ and the last $w_i$ to be revealed, both of which are now safe.

Writing $X_i$ for the number of vertices marked as safe in stage $i$, but setting $X_i=0$ if stage $i$ ends in failure, this process succeeds if and only if $\sum_{i=1}^kX_i=n$ before the first time $X_k=0$. Note that in each stage, no matter what has previously happened so long as we have not yet ended the process, with probability $1/2$ we reveal $v\not\in S$ and have $X_i=1$; with probability $(1/2)^{1+r}$ we end the process with failure; and with the remaining probability we mark as safe all of $w_1,\ldots,w_r$ and any neighbours of some $w_j$ not previously marked as safe (including $v$). In this final case we have $X_i\leq 2d-1$ except possibly in the first stage where we have $X_1\leq 2d$. Since $r\leq d$ it follows that $X_i\mid\mathcal{F}_{i-1}\leq_{st}X$ for each $i\geq 2$, where
\[X=\begin{cases}0&\text{with probability }\frac{1}{2^{d+1}}\\
1&\text{with probability }\frac12;\\
2d-1&\text{with probability }\frac{2^d-1}{2^{d+1}}.\end{cases}\]
Thus Corollary \ref{cor:game} applies with $z$ being the unique positive root of $2^dz+(2^d-1)z^{2d-1}=2^{d+1}$; since the LHS is clearly smaller for $z\leq 1$ we have $z>1$. This implies the required result since setting $y=z_d^{-1}$ in this equation gives $y=y_d$.
\end{proof}
Computation of $y_3$, together with Theorem \ref{thm:criss-cross}, gives the following bounds.
\begin{cor}We have $0.813<c_3<0.9781$.\end{cor}
Finally, we note that no upper bound on the degree of a regular graph that is weaker than a constant bound is sufficient to give a global bound on $c(G)$. 
\begin{prop}\label{large-degree}For any function $f(n)=\omega(1)$, there are graphs $G_n$ of order $n$ which are $d$-regular for $d\leq f(n)$ with $c(G_n)\to 1$. Also, for any fixed $d\geq 3$ there are $d$-regular graphs of every sufficiently large even order satisfying $D(G)>1-\eps(d)$, where $\eps(d)\to 0$ as $d\to\infty$.
\end{prop}
\begin{proof}Note that $C_n^k$ (that is, the graph obtained from the cycle $v_1v_2\cdots v_n$ by adding edges $\{v_iv_j:\abs{i-j}\leq k\}$ and $\{v_iv_j:\abs{n+i-j}\leq k\}$) is a $2k$-regular graph with $n$ vertices and contains a connected dominating set $S$ of order at most $n/2k$ consisting of $\{v_i:i\equiv 1\pmod{k}\}$. Any superset of $S$ is also connected, and so $N(C_n^k)\geq 2^{n-n/2k}$, giving $c(C_n^k)\geq \sqrt[2k]{1/2}=1-o(1)$ provided $k=\omega(1)$, and we can choose such $k=k(n)$ with $k\leq f(n)$. This also proves the second statement for even $d$. Similarly, the Cartesian product $C_n^k\square K_2$ is a $2k+1$-regular graph on $2n$ vertices with a connected dominating set of order at most $n/k$, extending the proof to odd $d$.\end{proof}

\section{Bounding the connected set density away from 1}\label{sec:q8}
In this section we turn our attention to Question \ref{q8}, which asks for an upper bound on the connected set density for graphs with minimum degree at least $3$. Our approach is in some sense complementary to that of Section \ref{sec:q5}, since our upper bound on the connected set density will follow from a lower bound on $c(G)$. We first need some straightforward inequalities.

\begin{lem}\label{binomial}For $n\in\zp$ and $\alpha\in(1/2,1)$, set $\mathcal S_\alpha=\{S\subseteq [n]:\abs{S}\geq\alpha n\}$. Then we have 
\[\abs{\mathcal S_\alpha}<\frac{\alpha}{2\alpha-1}\bigl(\alpha^{-\alpha}(1-\alpha)^{\alpha-1}\bigr)^n,\]
and the average order of elements of $\mathcal S_\alpha$ is less than $\alpha n+1+\frac{\alpha(1-\alpha)}{(2\alpha-1)^2}$.
\end{lem}
\begin{proof}
Write $k=\ceil{\alpha n}$, so that $\abs{\mathcal S_\alpha}=\binom nk+\binom{n}{k+1}+\cdots+\binom nn$. Note that for each $j\geq 0$ we have 
\[\binom n{k+j}=\binom nk\prod_{i=1}^j\frac{n-k-i+1}{k+i}<\binom nk\bfrac{n-k}{k}^j.\]
Since $n-k<k$ we have 
\[\abs{\mathcal S_\alpha}<\binom nk\sum_{j\geq 0}\bfrac{n-k}{k}^j=\binom nk\frac{k}{2k-n}.\]
Similarly, the total order of all sets in $\mathcal S_\alpha$ is 
\begin{align*}k\abs{\mathcal S_\alpha}+\sum_{j=0}^{n-k}j\binom n{k+j}&< k\abs{\mathcal S_\alpha}+\binom nk\sum_{j\geq 0}j\bfrac{n-k}{k}^j\\
&=k\abs{\mathcal S_\alpha}+\binom nk\frac{k(n-k)}{(2k-n)^2}.\end{align*}
Consequently, since $\abs{\mathcal S_\alpha}\geq\binom nk$, the average order of all these sets is below $k+\frac{k(n-k)}{(2k-n)^2}$. Note that for $n/2<x\leq n$, the expression $\frac{x(n-x)}{(2x-n)^2}$ is decreasing in $x$, and since $k\geq \alpha n$, we have $\frac{k(n-k)}{(2k-n)^2}\leq \frac{\alpha(1-\alpha)}{(2\alpha-1)^2}$. Since also $k<\alpha n+1$, the average order of sets in $\mathcal S_\alpha$ is less than $\alpha n+1+\frac{\alpha(1-\alpha)}{(2\alpha-1)^2}$, as required.

To complete the bound on $\abs{\mathcal S_\alpha}$, note that Stirling's approximations imply 
\[\binom nk\leq(k/n)^{-k}(1-k/n)^{k-n}\leq \bigl(\alpha^{-\alpha}(1-\alpha)^{\alpha-1}\bigr)^n\]
and that $\frac{k}{2k-n}\leq \frac{\alpha}{2\alpha-1}$.
\end{proof}
We now give a general upper bound on $D(G)$ in terms of $c(G)$.
\begin{thm}\label{too-many-sets}Fix $c\in(1/2,1)$ and let $G_n$ be an $n$-vertex graph with $c(G_n)\geq c$. Then $D(G_n)\leq \alpha+o(1)$ and $D(G_n)\geq1-\alpha-o(1)$, where $\alpha=\alpha(c)$ is the unique value in $(1/2,1)$ satisfying $\alpha^{-\alpha}(1-\alpha)^{\alpha-1}=2c$.\end{thm}
\begin{proof}Since $c(G_n)\geq c$, the graph $G_n$ has at least $(2c)^n$ connected sets. The average order of these sets is at most that of the $(2c)^n$ largest subsets of $V(G_n)$. Fix $\alpha''>\alpha'>\alpha$. By choice of $\alpha$ we have $\alpha'^{-\alpha'}(1-\alpha')^{\alpha'-1}<2c$ and so Lemma \ref{binomial} gives $\abs{\mathcal S_{\alpha'}}<(2c)^n$ for $n$ sufficiently large. Consequently the $(2c)^n$ largest subsets of $V(G_n)$ consist of $\mathcal S_{\alpha'}$ together with some other smaller subsets, so have smaller average order than $\mathcal S_{\alpha'}$. Also, again by Lemma \ref{binomial}, the average order of sets in $\mathcal S_{\alpha'}$ is less than $\alpha''n$ for $n$ sufficiently large. Consequently, for any $\alpha''>\alpha$ we have $D(G_n)<\alpha''$ for $n$ sufficiently large, \ie $D(G_n)\leq \alpha+o(1)$. Repeating the argument for the $(2c)^n$ smallest sets, noting that the sets of size at most $(1-\alpha')n$ are the complements of sets in $\mathcal S_{\alpha'}$, gives the corresponding lower bound.
\end{proof}

We combine this with known results that give lower bounds on the number of connected sets. Vince observed \cite{Vin17} that if $G$ has a spanning tree with $\ell$ leaves then $N(G)\geq2^\ell$, since any set containing all internal vertices of the spanning tree is necessarily connected, and used a result of Kleitman and West \cite{KW91} (which they attribute to Linial and Sturtevant) that any $n$-vertex graph with minimum degree at least $3$ has a spanning tree with at least $2+n/4$ leaves to give an exponential lower bound on $N(G)$. We use the following extension of Kleitman and West's result due to Karpov \cite{Kar14}.
\begin{thm}[Karpov, \cite{Kar14}]\label{Karpov}If a graph $G$ has $s$ vertices with degree $1$ or $3$ and $t$ vertices with degree at least $4$, then $G$ has a spanning tree with at least $\ceil{s/4+t/3+3/2}$ leaves.\end{thm}
Theorem \ref{Karpov}, together with the above observation, immediately gives an affirmative answer to \cite[Conjecture 1]{Vin17}, that any graph family where the proportion of vertices of degree at least $3$ is bounded away from $0$ has an exponential lower bound on $N(G)$, \ie $c(G)$ is bounded away from $1/2$. It also gives the following bound for graphs with no vertex of degree $2$; notice that in this case we can dispose of the $o(1)$ term entirely.
\begin{thm}Any non-trivial connected graph $G$ with no vertex of degree $2$ satisfies $D(G)<\alpha(2^{-3/4})<0.95831$.
\end{thm}
\begin{proof}Let $G$ have $n$ vertices; we may assume $n\geq 4$ since $D(K_2)=2/3$. For ease of notation write $\alpha=\alpha(2^{-3/4})$; note that $0.95<\alpha<1$ and so $\frac{\alpha}{2\alpha-1}<1.1$ and $\frac{\alpha(1-\alpha)}{2\alpha-1}<0.1$. 

By Theorem \ref{Karpov} there is a collection $\mathcal C_1$ of at least $2^{\ceil{n/4+3/2}}$ connected sets of average order at least $7n/8$. Let $\mathcal C_2$ be all other connected sets of $G$, and suppose that $D(G)>\alpha(2^{-3/4})$, \ie 
\begin{equation}\frac{\sum_{S\in \mathcal C_2}\abs S/n+2^{\ceil{n/4+3/2}}\times7/8}{\abs{\mathcal C_2}+2^{\ceil{n/4+3/2}}}>\alpha.\label{condition}\end{equation}
Note that if \eqref{condition} holds, then it still holds after adding a set of size at least $\alpha n$ to $\mathcal C_2$ or after deleting a set of size less than $\alpha n$ from $\mathcal C_2$. Thus we must also have
\[\frac{\sum_{S\in \mathcal S_\alpha}\abs S/n+2^{\ceil{n/4+3/2}}\times7/8}{\abs{\mathcal S_\alpha}+2^{\ceil{n/4+3/2}}}>\alpha.\]

Now by Lemma \ref{binomial} and choice of $\alpha$ we have $\sum_{S\in \mathcal S_\alpha}\abs S/n\leq \abs{S_\alpha}(\alpha+1.1/n)$ and $\abs{S_\alpha}\leq 1.1\times 2^{n/4}$. Consequently we must have $1.1(\alpha+1.1/n)+2^m\times 7/8>(1.1+2^m)\alpha$, where $m=\ceil{n/4+3/2}-n/4$, which is false for every $n\geq 4$.
\end{proof}

Note in the proof above that we start from a large group $\mathcal C_1$ of connected sets with average order about $7n/8$, and $7/8$ is much less than the upper bound on $D(G)$ we obtain. However, the fact that the connected sets contain a large subcollection of small average size does not immediately help us, since for any fixed $\eps>0$ the family of subsets of order at least $(\alpha(2^{-3/4})-\eps/2)n$ is exponentially larger than $\mathcal C_1$, and so for $n$ sufficiently large there is a collection of subsets containing $\mathcal C_1$ but with average order exceeding $(\alpha-\eps)n$.

This analysis does not take into account any structural properties of the collection of all connected sets; it is clearly not possible for these sets and no others to be connected. One might therefore be tempted to conjecture that if an $n$-vertex graph $G$ has a spanning tree with at least $Ln$ leaves for some constant $L$ then $D(G)\leq 1-L/2+o(1)$; that is, the connected sets we construct from this tree are asymptotically at least as large as the remaining connected sets. However, the following example shows that this is not the case.

\begin{prop}For each $n$ divisible by $6$ there is an $n$-vertex graph $H_n$ containing a spanning tree with $n/2+2$ leaves but for which $D(H_n)=41/54-o(1)=3/4+1/108-o(1)$.
\end{prop}
\begin{proof}Let $H$ be the graph with vertex set $\{u,v,w,x,y,z\}$ and edge set $\{uv,\allowbreak vw,\allowbreak vx,\allowbreak vy,\allowbreak vz,\allowbreak ux,\allowbreak xy,\allowbreak yz,\allowbreak zw\}$, and let $H_n$ be the graph consisting of $n/6$ copies of $H$ linked cyclically, with an edge from each copy of $w$ to the next copy of $u$; see Figure \ref{hexagons}. Here there is a spanning tree with $n/2+2$ leaves (a section of which is shown in bold); however, we may argue as in Theorem \ref{thm:criss-cross} to estimate $D(G)$. For a connected set $S\subseteq V(H_n)$ and a given copy of $H$ in $H_n$, we say $S$ is good for that copy if the restriction of $S$ to that copy contains a $u$-$w$ path. Note that there are $9$ different ways $S$ can be good for a copy, since it can intersect that copy in $\{u,x,y,z,w\}$ or any subset containing $\{u,v,w\}$. If $S$ is bad for two copies of $H$ then $S$ intersects at most one of the components left when these copies are removed. Consequently if $S$ is good for $k$ copies then these copies form an interval, and there are $\Theta(9^kn)$ connected sets which are good for exactly $k$ copies. It follows that almost all connected sets are good for almost all copies, and since the connected sets which are good for a given copy intersect that copy in $41/9$ vertices on average, this gives the desired result.\end{proof}
\begin{figure}
\centering
\begin{tikzpicture}
\foreach \y in {-3,0,3,6}{%
\foreach \x in {30,90,...,330}{\draw[fill] (\y,0)+(\x:1) circle (0.05);
\draw (\y,0)++(\x:1) -- ++(\x+120:1);
\draw[very thick] (\y,0)++(\x:1) -- (\y,1);}
\draw[very thick] (\y,0)++(30:1) -- ++(1,0);
\draw[very thick] (\y,0)++(150:1) -- ++(-1,0);}
\path (-4,0)++(150:1) node[anchor=north]{$\cdots$};
\path (7,0)++(30:1) node[anchor=north]{$\cdots$};
\path (-3,0)++(150:1) node[anchor=south]{$u$};
\path (-3,0)++(90:1) node[anchor=south]{$v$};
\path (-3,0)++(30:1) node[anchor=south]{$w$};
\path (-3,0)++(210:1) node[anchor=north]{$x$};
\path (-3,0)++(270:1) node[anchor=north]{$y$};
\path (-3,0)++(330:1) node[anchor=north]{$z$};
\end{tikzpicture}
\caption{Part of an $n$-vertex graph having a spanning tree with more than $n/2$ leaves but having connected set density $3/4+1/108-o(1)$.}\label{hexagons}
\end{figure}
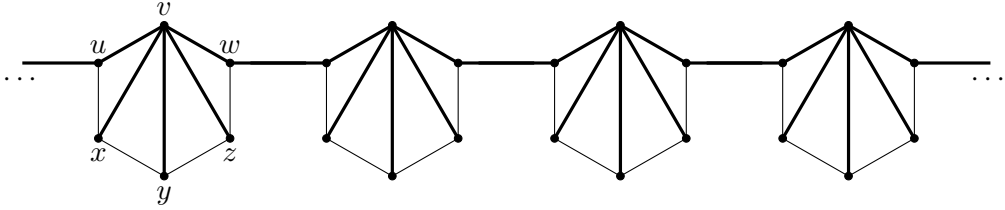

We can also show that many vertices of degree $2$ are required for high connected set density, which generalises a result of Jamison \cite[Theorem 6.2]{Jam83} for trees.
\begin{thm}If $G_n$ is a sequence of non-trivial connected graphs such that $G_n\to 1$ then $v_2(G_n)/\abs{G_n}\to 1$, where $v_2(G_n)$ is the number of vertices of degree $2$.\end{thm}
\begin{proof}Suppose not, so that there is such a sequence with $v_2(G_n)<x\abs{G_n}$ for some $x<1$; since $D(G_n)\to 1$ and each $G_n$ is non-trivial we have $\abs{G_n}\to\infty$. Theorem \ref{Karpov} gives a spanning tree with at least $(1-x)\abs{G_n}/3$ leaves, and so $c(G_n)\geq 2^{-(x+2)/3}>1/2$. Theorem \ref{too-many-sets} therefore implies $D(G_n)\leq \alpha(2^{-(x+2)/3})+o(1)$, a contradiction.
\end{proof}
Finally, we consider the effect of increasing the minimum degree.
\begin{thm}If $G_n$ is a sequence of connected graphs with $\delta(G_n)\to\infty$ then $D(G_n)\to1/2$.\end{thm}
\begin{proof}Suppose not, and take such a sequence with $D(G_n)>x$ for some $x>1/2$ (the low-density case where $D(G_n)<1-x$ is similar). Choose $c<1$ such that $1/2<\alpha(c)<x$; this is possible since $\alpha^{-\alpha}(1-\alpha)^{\alpha-1}$ is bounded away from $2$ if $x\leq\alpha\leq 1$. By Theorem \ref{too-many-sets}, we must have $c(G_n)<c$ for $\abs{G_n}$ sufficiently large, and in particular if $\delta(G_n)\geq k_1$ for some fixed $k_1$. 

However, by \cite[Theorem 5]{KW91}, there is some constant $k_2$ such that for $k\geq k_2$ any graph $G$ with minimum degree at least $k$ has a spanning tree with at least $\bigl(1-3\frac{\log k}{k}\bigr)\abs{G}$ leaves. As a consequence such a graph has $c(G)>2^{-3\frac{\log k}{k}}$. For $k\geq k_3$ we have $2^{-3\frac{\log k}{k}}>c$, and for $n$ sufficiently large $\delta(G_n)\geq\max\{k_1,k_2,k_3\}$, giving a contradiction.
\end{proof}

\section{Open problems}
This work leaves several natural open problems. First, to improve either of the bounds obtained on $c_3$. In particular, is the criss-cross prism asymptotically optimal? Concerning the corresponding quantity for $d$-regular graphs, $c_d$, we know by Proposition \ref{large-degree} that $c_d\to 1$ as $d\to\infty$, but how quickly does $c_d$ approach $1$?

Secondly, we might ask about the corresponding upper bounds on $D(G)$. What is the asymptotic maximum density $\limsup_{G\in\mathcal G_3} D(G)$? Does the answer change if we consider graphs of minimum degree $3$, rather than $3$-regular graphs? The best known lower bound for either case is $5/6$: Vince \cite[Theorem 3.1]{Vin21} evaluated the connected set density of the necklace graphs precisely, and this is $5/6-o(1)$. Note that this is greater than the corresponding upper bound for series-reduced trees.

Finally, we turn to lower bounds. Here we expect graphs of minimum degree $3$ to satisfy the same lower bound as series-reduced trees, that $D(G)\geq 1/2$; this was conjectured by Vince \cite[Conjecture 7]{Vin21}.

\end{document}